\documentclass[11pt]{amsart}
\usepackage[english]{babel}
\usepackage[T1]{fontenc} \usepackage[latin1]{inputenc}
\usepackage{graphics,enumerate,amssymb,textcomp}
\usepackage{color,epsfig}

\numberwithin{equation}{section}

\theoremstyle{plain} \newtheorem{thm}{Theorem}[section]
\newtheorem{lemma}[thm]{Lemma} 
\newtheorem{proposition}[thm]{Proposition}
\newtheorem{corollary}[thm]{Corollary}
 \theoremstyle{definition}
 \newtheorem{remark}[thm]{Remark}

\def\P {\mathbb {P}}    \def\Q {\mathbb {Q}}
 \def\phi{\varphi} \def\D{\Delta} 
\def\tpar{/ \! /}  \def\R{\mathbb{R}} \def\E{\mathbb{E}}
 \def\1{\mathbbm{1}}
\def\W {\textbf{W}}
\def\mathpal#1{\mathop{\mathchoice{\text{\rm #1}}%
    {\text{\rm #1}}{\text{\rm #1}}%
    {\text{\rm #1}}}\nolimits} 

\newcommand\Ric{\mathpal{Ric}}
\newcommand\Id{\mathpal{Id}}

\newcommand\trace{\mathpal{trace}}

\title{Stochastic proof of upper bound for the heat kernel coupled with geometric flow, and Ricci flow}

\author [K.A. Coulibaly-Pasquier]{Kol\'eh\`e A. Coulibaly-Pasquier}

\address{Nancy}
\email{kolehe.coulibaly@iecn.u-nancy.fr}

\begin{document}

\maketitle

\begin{abstract}
 We  give a proof of Gaussian upper bound for the heat kernel coupled with the Ricci flow. Previous proofs by Lei Ni \cite{Ni} use Harnack
 inequality and doubling volume property, also the recent proof by  Zhang and Cao \cite{Z-C} uses Sobolev type inequality that is conserved along Ricci flow. 
We will use a horizontal coupling of curve \cite{ATC2} Arnaudon Thalmaier, C. , in order to generalize  Harnack inequality with power -for inhomogeneous heat equation  - introduced by  
 F.Y Wang.  In the case of Ricci flow, we will derive on-diagonal bound of the Heat kernel along Ricci flow
 ( and also for the usual Heat kernel on complete Manifold).

\end{abstract}

\section{Coupling and Harnack inequality with power}
In the first part of this section, we will focus on the operator of the type $L_t := \frac12 \D_{g(t)} $,
 where $\D_{g(t)}$ is the Laplace operator associated to a time dependent family of metrics  $ g(t)$. We will suppose that all considered $g(t)$-Brownian motion  is non explosive.
 For example when the family of metric comes from the backward Ricci flow, this have been proved in \cite{K.P.}.
%
 Let $\gamma(t)$ be a $C^{1} (M) $ geodesic curve such that $\gamma (0) = x $ and $\gamma (1)=y$ and $X_t (u)$ be the the horizontal $L(t)$-diffusion $C^1$ path space 
 in $C^1([0,T],M)$ over $X^0$, started at $u \mapsto \gamma (\frac{u}{T})$. Let $X_t(x) $ be a $g(t)$-Brownian motion that start at $x$, $\tpar_{0,t}$ the $g(t)$ parallel transport,
 and  $\W_t$ the damped parallel transport
that satisfies the following Stratonovich covariant equation:
$$ *d ((\tpar_{0,t})^{-1} (\textbf{W}_{0,t})) = -\frac12 (\tpar_{0,t})^{-1} ( \Ric_{g(t)} - \partial_{t}(g(t)))^{\# g(t)}(\textbf{W}_{0,t}) \,dt $$
with
$$ \textbf{W}_{0,t} : T_{x}M \longrightarrow T_{X_{t}(x)}M , \textbf{W}_{0,0}=\Id_{T_{x}M}.$$

\begin{proposition}
 The process  $X_t (\gamma(\frac{t}{T}))$ satisfies the following stochastic differential equation :
$$ d^{\nabla_t} ( X_{.} (\gamma(\frac{.}{T})))_t = P_{0,\frac{t}{T}}^{t, X_t( . )} d^{\nabla_t} X_t^0(x) + \frac{1}{T} \W(X_{.}( \gamma (\frac{t}{T})) )_t \dot{\gamma}(\frac{t}{T}) \, dt$$
\end{proposition}
\begin{proof}
 We pass to the Stratonovich differential and obtain the following chain rule formula:
 $$ *d (X_{.} ( \gamma (\frac{.}{T}))  )_{t_0} = *d (X_{.} (\gamma (\frac{t_0}{T})))_{t_0} +  \frac{d X_{t_0}(\frac{t}{T})} {d t}_{t = t_0} dt_{0} .$$
We use Theorem 3.1 in \cite{ATC2} to identify the last term of the right hand side:
 $$\frac{d X_{t_0}(\frac{t}{T})} {d t}_{t = t_0} = \frac{1}{T} \W(X_{.} (\gamma (\frac{t_0}{T}))  )_{t_0} (\dot{\gamma}(\frac{t_0}{T})) .$$ 
Now we come back to the It\^o differential equation using the following relation:
 $$ d^{\nabla_t} Y_t = \tpar_t ( d \int_0 ^t \tpar_s^{-1} *d Y_s) , $$
and we obtain 
$$\begin{aligned}
 d^{\nabla_{t_{0}}} ( X_{.} (\gamma(\frac{.}{T})))_{t_{0}} &=   \tpar_{t_0} \big( d \int_0 ^{t_0} \tpar_s^{-1} *d (X_{.} (\gamma (\frac{t_0}{T})))_{s} + 
 \frac{1}{T} \tpar_s^{-1} \W(X_{.}(\gamma (s))  )_{s} (\dot{\gamma}(\frac{s}{T})) ds \big) \\
                                                    &= d^{\nabla_{t_0}} ( X_{.} (\gamma(\frac{t_0}{T})))_{t_0} + \frac{1}{T} \W(X_{.}( \gamma (\frac{t_0}{T})) )_{t_0} \dot{\gamma}(\frac{t_0}{T}) .
\end{aligned}  $$

We use again Theorem 3.1 in \cite{ATC2} to  identify  $$d^{\nabla_{t_0}} ( X_{.} (\gamma(\frac{t_0}{T})))_{t_0} = P_{0,\frac{t_0}{T}}^{t_0, X_{t_0}( . )} d^{\nabla_{t_0}} X_{t_0}^{0}(x) .$$
 
\end{proof}

 Let $$N_t := - \frac{1}{T} \int_0^t \langle  P_{0,\frac{s}{T}}^{s, X_s( . )} d^{\nabla_s} X_s^0(x) , \W(X_{.}( \gamma (\frac{s}{T})) )_s \dot{\gamma}(\frac{s}{T})   \rangle_{g(s)},$$
 $$R_t := \exp \big( N_t -\frac12 \langle N \rangle_t  \big).$$
In many situations the Novikov's criterion will be satisfied, so we could expect $R_t$ to be a martingale. 
Define a new probability measure as :
$$ \Q := R_T \P .$$
 
\begin{proposition}\label{croise}
Suppose that Novikov's criterion is satisfied for $N_t$. Then under $\Q$, the process $X_{t}(\gamma (\frac{t}{T}) $ is a $L_t$-diffusion that starts at $x$,
 that finishes at $X_T(y)$.
\end{proposition}
\begin{proof}

One could directly apply Girsanov's theorem. We prefer here to give a direct proof.
Let $f \in C^{2}(M,\R)$. Recall that $R_t $ is a $\P$-martingale, and $ P_{0,\frac{t}{T}}^{t, X_t( . )} $ is an isometry for the metric $g(t)$. Use It\^o formula to compute :
$$
\begin{aligned}
& d( R_t f(X_{t}(\gamma (\frac{t}{T})))) \\
&= \frac12 R_t \D_t f(X_{t}(\gamma (\frac{t}{T})))dt + d M^{\P}_t \\
\end{aligned}
 ,$$ 
where $ M^{\P}_t $  is a martingale for $\P$. Also 

$$R_t f(X_{t}(\gamma (\frac{t}{T}))) - \frac12 R_t \int_0^t \D_sf(X_{s}(\gamma (\frac{s}{T}))) ds = $$  
$$  R_t f(X_{t}(\gamma (\frac{t}{T}))) - \frac12  \int_0^t R_s \D_sf(X_{s}(\gamma (\frac{s}{T}))) ds +  \tilde{M}^{\P}_s       .         $$
Where $ \tilde{M}^{\P}_s$ is a $\P$-martingale.
Use the fact that $U_t$ is a $\Q $-martingale if and only if $R_t U_t$ is a $\P$-martingale.
So $ f(X_{t}(\gamma (\frac{t}{T}))) - \frac12 \int_0^t \D_sf(X_{s}(\gamma (\frac{s}{T}))) ds$ is a $\Q$ martingale i.e. $ X_{t}(\gamma (\frac{t}{T}))$ is a $L_t$ diffusion under the probability $\Q. $
It is clear that it finishes at $X_T(y)$, so $X_{t}(\gamma (\frac{t}{T})))$ can be seen as a coupling between two $L_t$ diffusions that started at different point up to  one change of probability.
  
\end{proof}

 Let $ \alpha_{i,j} (t) $ be a family of symmetric $2$-tensors on $M$. We will consider the following heat equation coupled with a geometric flow.

\begin{equation}\label{heat}
  \left\{\begin{array}{l} 
  \partial_t g_{i,j} = \alpha_{i,j} (t) \\
  \partial_{t}f(t,x) = \frac12 \D_{t} f(t,x) \\
   f(0,x)=f_{0}(x),
   \end{array} \right .
\end{equation}

\begin{remark}
 Such flow, have been widely investigated in the literature. Let us mention the following situation:
\begin{itemize}
 \item The most famous  case is when $\alpha_{i,j} (t) := 0 $, this is the case of constant metric and equation \ref{heat} is the usual heat equation in $M$.
 \item   $\alpha_{i,j} (t):= - \Ric_{i,j} (g(t)) $, that is the Ricci flow.
 \item  One can also consider $\alpha_{i,j} (t):= -2h H_{i,j} (g(t))$, where $H_{i,j} (g(t)) $  is the second fundamental form relatively to the metric $g(t)$, and $h$ is the mean curvature, when the family of metric comes from the mean curvature flow.   
\end{itemize}
In all these cases  a notion of $ g(t) $-Brownian motion, i.e. a $\D_t$ diffusion,  parallel transport, and damped parallel transport  has been given in \cite{metric,ATC1}.
\end{remark}

Let $T_{c}$ be the maximal life time of geometric flow $g(t)_{t \in[0,T_{c}[}$. For all $T < T_{c} $, let  $X^{T}_{t}$ be a $g(T-t)$-Brownian motion and $\tpar^T_{0,t}$ the associated parallel transport.
In this case, for a solution $f(t,.)$ of (\ref{heat}), $ f(T-t,X^{T}_{t}(x))$ is a local martingale for any $x \in M$. So the following
 representation holds for the solution :
  $$ f(T,x) := \E_x[ f_0(X_T^T)] . $$

We introduced a further subscript $T$ referring to the fact that a time reversal step is involved.

Let $\textbf{W}_{0,t}^{T}$ be the damped parallel transport along the $g(T-t)$-Brownian motion.Let us recall the covariant differential equation satisfied by this damped parallel transport \cite{metric} :
$$ *d ((\tpar^{T}_{0,t})^{-1} (\textbf{W}_{0,t}^{T})) = -\frac12 (\tpar^{T}_{0,t})^{-1} ( \Ric_{g(T-t)} - \partial_{t}(g(T-t)))^{\# g(T-t)}(\textbf{W}_{0,t}^{T}) \,dt $$
with
$$ \textbf{W}_{0,t}^{T} : T_{x}M \longrightarrow T_{X_{t}^{T}(x)}M , \textbf{W}_{0,0}^{T}=\Id_{T_{x}M}.$$

All the over subscript $T$ we will mean that the family of metrics is $g(T-t)$.

\begin{proposition}\label{explosion}
 Suppose that there exist $\overline {\alpha} ,   \underline{\alpha}  \ge 0 $ and  $\overline{K},  \underline {K}\ge 0 $ such that :
  $$ -  \underline{\alpha}  g(t)    \le \alpha (t) \le  \overline {\alpha}  g(t) ,$$
  $$ -  \underline{K} g(t)   \le \Ric(t) \le    \overline {K} g(t) ,$$
then the $g(t)$-Brownian motion, and the $g(T-t)$-Brownian motion  does not explode before the time $T_c$.
\end{proposition}
\begin{proof}
 This is a sufficient condition but it is far from being a necessary one, for the process to don't explode. 
 Use the It\^o formula for $ d_t (x_0, X_t)$, the comparison theorem of the laplacian of the distance function, and 
 the comparison theorem of stochastic differential equation. 
\end{proof}

\begin{remark}

 For the backward Ricci flow, the $g(t)$-Brownian motion does not explode \cite{K.P.}, but the condition of the sufficient existence of the Ricci flow
 in complete Riemannian manifolds  as given by Shi \cite{Shi} theorem 1.1, that is the boundedness of the initial Riemannian tensor (for the metric $ g(0)$)
 also gives a bound of the $ Ric $ tensor  along the flow (for bounded time). So the conditions for non explosion of the $g(t) $-Brownian motion given
 in the above proposition
will be satisfied if the initial metric satisfy Shi condition for the complete manifolds.
\end{remark}

\begin{proposition} \label{controleR} Suppose that there exist  $C \in \R $ such that in a matrix sens:
 $$ \Ric_{g(t)}  - \alpha(t) \ge C g(t) .$$
Then  $R_t$ is a martingale, and for $ \beta \ge 1$
$$ \E[R_t^{\beta}] \le e^{\frac12\beta(\beta-1) \frac{d_0^2(x,y)}{T^2}  \frac{1- e^{-Ct}}{C}} .$$ 
Moreover  suppose that there exist  $\tilde{C} \in \R $ such that in a matrix sens: $$ \Ric_{g(t)}  + \alpha(t) \ge \tilde{C} g(t) $$
then  $R^T_t$  is a martingale and for $ \beta \ge 1$
$$ \E[(R^T_t)^{\beta}] \le e^{\frac12\beta(\beta-1) \frac{d_T^2(x,y)}{T^2}  \frac{1- e^{-\tilde{C}t}}{\tilde{C}}} .$$ 
\end{proposition}

\begin{proof}
Let $ v \in T_x M $.  Then we use the isometry property of the parallel transport i.e. $\tpar_s : (T_xM,g(0)) \mapsto (T_{X_s(x)}M,g(s)) $ to deduce
 $$\begin{aligned}
   *d \langle \W(X_{.}(x))_s v &, \W(X_{.}(x))_s v  \rangle_{g(s)} =   *d  \langle \tpar^{-1}_s\W(X_{.}(x))_s v, \tpar^{-1}_s \W(X_{.}(x))_s v  \rangle_{g(0)} \\
     &=  2  \langle *d \tpar^{-1}_s\W(X_{.}(x))_s v, \tpar^{-1}_s \W(X_{.}(x))_s v  \rangle_{g(0)} \\
       &= 2  \langle \tpar_s *d \tpar^{-1}_s\W(X_{.}(x))_s v,  \W(X_{.}(x))_s v  \rangle_{g(s)} \\
 &= -  \langle   ( \Ric_{g(s)} - \partial_{s}(g(s)))^{\# g(s)}(\W(X_{.}(x))_{s}v), \W(X_{.}(x))_{s}v \rangle_{g(s)}  \,ds \\ 
&\le - C \parallel \W(X_{.}(x))_{s}v \parallel^2 \, ds
   \end{aligned}
 $$ 
By Gronwall's lemma we get 
$$ \parallel \W(X_{.}(x))_{s}v \parallel_{g(s)}  \le e^{-\frac12 Cs} \parallel v \parallel_{g(0)} $$
Recall  that 
$N_t := - \frac{1}{T} \int_0^t \langle  P_{0,\frac{s}{T}}^{s, X_s( . )} d^{\nabla_s} X_s^0(x) , \W(X_{.}( \gamma (\frac{s}{T})) )_s \dot{\gamma}(\frac{s}{T})   \rangle_{g(s)} $,
and $ P_{0,\frac{s}{T}}^{s, X_s( . )}$ is a $g(s)$ isometry and $ d^{\nabla_s} X_s^0(x)= \tpar_s e_i dw^i$ where $w$ is a $\R^n$-Brownian motion, and $(e_i)_{i=1..n}$
is an orthonormal basis of $T_xM $. Then

$$\begin{aligned}
 \langle N \rangle_t &= \frac{1}{T^2} \int_0^t \parallel \W(X_{.}( \gamma (\frac{s}{T})) )_s \dot{\gamma}(\frac{s}{T}) \parallel^2_{g(s)} \, ds \\
  &\le \frac{1}{T^2} \int_0^t e^{- Cs} \parallel  \dot{\gamma}(\frac{s}{T}) \parallel^2_{g(0)} \, ds \\
 &\le \frac{1}{T^2} d^2_0(x,y) \int_0^t e^{- Cs} \, ds 
\end{aligned}
$$
So by Nokinov's criterion, $R_t $ is a martingale.
Let $\beta \ge 1$,
$$\begin{aligned}
    \E[R_t^{\beta}] &=  \E[ e^{\beta N_t - \frac{\beta}{2} \langle N\rangle_t}]\\
 &= \E[ e^{\beta N_t - \frac{\beta^2}{2} \langle N\rangle_t} e^{ \frac{\beta( \beta-1)}{2} \langle N\rangle_t}] \\
     & \le e^{\frac12\beta(\beta-1) \frac{d_0^2(x,y)}{T^2}  \frac{1- e^{-Ct}}{C}} .
  \end{aligned} $$

By the same computation we have 
 $$\begin{aligned}
 \langle N^T \rangle_t &= \frac{1}{T^2} \int_0^t \parallel \W^T(X_{.}( \gamma (\frac{s}{T})) )_s \dot{\gamma}(\frac{s}{T}) \parallel^2_{g(T-s)} \, ds \\
  &\le \frac{1}{T^2} \int_0^t e^{- \tilde{C} s} \parallel  \dot{\gamma}(\frac{s}{T}) \parallel^2_{g(T)} \, ds \\
   &\le \frac{1}{T^2} d^2_T(x,y) \int_0^t e^{- \tilde{C}s} \, ds.
\end{aligned}
$$
Thus $R^T_t $ is a martingale. Given $\beta \ge 1$ we have similarly,
$$\begin{aligned}
    \E[(R^T_t)^{\beta}] &\le  e^{\frac12\beta(\beta-1) \frac{d_T^2(x,y)}{T^2}  \frac{1- e^{-\tilde{C}t}}{\tilde{C}}} .
  \end{aligned} $$
\end{proof}

\begin{remark}
 In the case of Ricci flow, $ \partial_t g(t) = - Ric_{g(t)} $, then  $ \partial_t g(T-t) =  Ric_{g(T-t)} $ so the process $X_t^T (x) $
 does not explode (we do not need proposition \ref{explosion}, but \cite{K.P.}) and
 the condition of the above proposition is satisfied with  $\tilde{C} =0 $ and $$\E[(R^T_T)^{\beta}] \le  e^{\frac12\beta(\beta-1) \frac{d_T^2(x,y)}{T}  } . $$
\end{remark}

We are now ready to give the Harnack inequality with power. Let $f$ be a solution of \eqref{heat} and let $P_{0,T}$ be the inhomogeneous heat kernel associated to \eqref{heat}, i.e.
$$P_{0,T} f_0(x) := f(T,x) = \E_x[ f_0(X_T^T)]. $$
  
\begin{thm}\label{harnach}
 Suppose that the $g(T-t)$-Brownian motion  $ X_t^T $ does not explode, and that the process $R_t^T $ is a martingale. Then :
 for all $ \alpha > 1$ and $ f_0 \in C_{b}(M)$ we have :

$$ \mid P_{0,T}f_0 \mid^{\alpha}  (x) \le  \E[(R^T_t)^{\frac{\alpha}{\alpha-1}}]^{\alpha-1}  P_{0,T} \mid f_0 \mid ^{\alpha} (y) .$$
Moreover if there exists  $\tilde{C} \in \R $ such that in a matrix sens: $$ \Ric_{g(t)}  + \alpha(t) \ge \tilde{C} g(t)  $$ then  we have:
$$  \mid P_{0,T}f_0 \mid^{\alpha}  (x) \le e^{\frac{\alpha}{2(\alpha-1)} \frac{d_T^2(x,y)}{T^2}  \frac{1- e^{-\tilde{C}T}}{\tilde{C}}} P_{0,T} \mid f_0 \mid ^{\alpha} (y) . $$
\end{thm}
\begin{proof}
We write $\tilde{X}^T_t := X^T_t (\gamma(\frac{t}{T})) $, and use \ref{croise}, and Holder inequality.
 $$ \begin{aligned}
     \mid P_{0,T}f_0 \mid^{\alpha}  (x)  &= \mid  \E^{\mathbb{Q}}[f_0 (\tilde{X}^T_T) ]\mid^{\alpha}  \\
  &= \mid \E^{\mathbb{P}}[R_T^T f_0 (\tilde{X}^T_T) ] \mid^{\alpha} \\
&\le  \E^{\mathbb{P}}[(R_T^T)^{\frac{\alpha}{\alpha -1 }}  ]^{\alpha -1}   \E^{\mathbb{P}} [\mid f_0 \mid ^{\alpha}(\tilde{X}^T_T) ] \\
 &=  \E^{\mathbb{P}}[(R_T^T)^{\frac{\alpha}{\alpha -1 }}  ]^{\alpha -1}   \E^{\mathbb{P}}_y [\mid f_0 \mid ^{\alpha}(X^T_T) ] \\
 &=  \E^{\mathbb{P}}[(R_T^T)^{\frac{\alpha}{\alpha -1 }}  ]^{\alpha -1} P_{0,T} \mid f_0 \mid ^{\alpha} (y) . \\
    \end{aligned}
$$
The last part in the theorem is an application of  proposition \ref{controleR}.
\end{proof}
We will denote by $\mu_t $ the volume measure associated to the metric $g(t) $, and for $A$ a Borelian, $V_t(A) := \int_A 1 \, d\mu_t $, and $B_t(x,r) $
 the ball for the metric $g(t)$ of center $x$ and radius $r$. .
\begin{corollary}\label{corol}
 Suppose that the hypothesis of theorem \ref{harnach} is satisfied, and that there exist $\tilde{C} \in \R $ such that
  $ \Ric_{g(t)}  + \alpha(t) \ge \tilde{C} g(t) . $
Let  $f_0 \in L^{\alpha} (\mu_0) $. Moreover suppose that there exists a function  $\tau :[ 0,T] \mapsto \R $ such that :
              $$  \frac12 \trace_{g(t)} (\alpha (t)) (y)  \le \tau(t) ,\quad \forall (t,y) \in [0,T]\times M \quad $$
then 
              $$  \mid P_{0,T}f_0   \mid   (x)  \le \frac{e^{\frac{\int_0^T \tau(s)\, ds   + 1}{\alpha}}}{\big( V_T(B_T(x,\sqrt{\frac {2(\alpha -1)T^2}{\alpha (\frac{1- e^{-\tilde{C}T}}{\tilde{C}} )} } ) )   \big)^{\frac{1}{\alpha}}} \parallel f_0\parallel_{L^{\alpha} (\mu_0) } .$$
\end{corollary}
\begin{proof}

 If $f_0 \in C_b(M) \cap L^{\alpha}(\mu_0)$ we apply theorem \ref{harnach} and get :
$$  \mid P_{0,T}f_0 \mid^{\alpha}  (x) \le e^{\frac{\alpha}{2(\alpha-1)} \frac{d_T^2(x,y)}{T^2}  \frac{1- e^{-\tilde{C}T}}{\tilde{C}}} P_{0,T} \mid f_0 \mid ^{\alpha} (y) . $$
We integrate both sides  along the ball $B_T \big(x,\sqrt{\frac {2(\alpha -1)T^2}{\alpha (\frac{1- e^{-\tilde{C}T}}{\tilde{C}} )} } \big) $, with respect to the measure $\mu_t  $, in $y$  and obtain :
 $$\begin{aligned}
  V_T(B_T\big(x,\sqrt{\frac {2(\alpha -1)T^2}{\alpha (\frac{1- e^{-\tilde{C}T}}{\tilde{C}} )} } \big) )  \mid P_{0,T}f_0 \mid^{\alpha}  (x) & \le e \int_{B_T\big(x,\sqrt{\frac {2(\alpha -1)T^2}{\alpha (\frac{1- e^{-\tilde{C}T}}{\tilde{C}} )} } \big)} P_{0,T} \mid f_0 \mid ^{\alpha} (y) \, d\mu_T(y) \\
 &\le  e \int_{M} P_{0,T} \mid f_0 \mid ^{\alpha} (y) \, d\mu_T(y).   
   \end{aligned}
  $$ 
We have that $ \frac{d}{dt} \mu_t(y) = \frac12 \trace_{g(t)} (\alpha (t)) (y) d\mu_t(y) $, and by Stokes theorem we have :
$$
\begin{aligned}
\frac{d}{dt} \int_{M} P_{0,t} \mid f_0 \mid ^{\alpha} (y) \, d\mu_t(y) &=  \int_{M} P_{0,t} \mid f_0 \mid ^{\alpha} (y) \,\frac{d}{dt}  d\mu_t(y) \\
&\le    \tau(t) \int_{M} P_{0,t} \mid f_0 \mid ^{\alpha} (y) \, d\mu_t(y) .  \\
\end{aligned}
 $$
 We deduce that :
$$ \int_{M} P_{0,t} \mid f_0 \mid ^{\alpha} (y) \, d\mu_t(y) \le e^{ \int_0^t \tau(s) \, ds }  \parallel f_0\parallel^{\alpha}_{L^{\alpha} (\mu_0) } .$$
So for $f_0 \in C_b(M) \cap L^{\alpha}(\mu_0)$  we have
 $$  \mid P_{0,T}f_0   \mid   (x)  \le \frac{e^{\frac{\int_0^T \tau(s)\, ds   + 1}{\alpha}}}{\big( V_T(B_T(x,\sqrt{\frac {2(\alpha -1)T^2}{\alpha (\frac{1- e^{-\tilde{C}T}}{\tilde{C}} )} } ) )   \big)^{\frac{1}{\alpha}}} \parallel f_0\parallel_{L^{\alpha} (\mu_0) } , $$
and we conclude by a classical density argument  that the same inequality is true for$f_0 \in L^{\alpha}(\mu_0)$  .
\end{proof}

\begin{corollary}
If the family of metric comes from the Ricci flow and  if 
$$(\tau(t))= - \frac12  \inf_{y \in M}R(t,y)  < \infty , \forall t \in [0,T]$$ then we have 
$$ \mid P_{0,T}f_0   \mid   (x)  \le \frac{e^{\frac{\int_0^T \tau(s) \, ds  + 1}{\alpha}}}{\big( V_T(B_T(x,\sqrt{\frac {2(\alpha -1)T}{\alpha } })  )   \big)^{\frac{1}{\alpha}}} \parallel f_0\parallel_{L^{\alpha} (\mu_0) } .$$
If  $\inf_{x \in M} \big( V_T(B_T(x,\sqrt{\frac {2(\alpha -1)T}{\alpha } })  )   \big) =: C_T > 0  $ then as an linear operator :
$$ \| P_{0,T}  \|_{L^{\alpha}(\mu_0) \mapsto L^{\infty} (\mu_0)} \le \frac{e^{\frac{\int_0^T \tau(s)\, ds   + 1}{\alpha}} }{C_T^{\frac{1}{\alpha}}}$$
\end{corollary}


\section{Non symmetry of the inhomogeneous heat kernel, and heat kernel estimate}
Let $\frac{\partial}{\partial_t} g(t) := \alpha(t) $ where $\alpha $ is a time dependent symmetric $2$-tensor; and consider 
 $L_{t,x}:= -\frac{\partial}{\partial_t} + \frac12 \D_{g(t)} $. Let $x,y \in M  $ and $0<\tau \le \sigma \le t $ and denote by $ P(x,t ,y,\tau) $ 
the fundamental solution of

\begin{equation} \label{heat_kernel}
  \left\{\begin{array}{l} 
   L_ {t,x} P(x,t,y,\tau) =0 \\
\lim_{t \searrow \tau} P(x,t,y,\tau) = \delta_{x}
   \end{array} \right .
\end{equation}
Recall that \cite{metric}:

$$ X^t_t(x) \overset{\mathcal{L}}{=} P(x,t,y,\tau) \, d \mu_0(y) $$

Let $ v, u \in C^{1,2}(\R,M)$, that is differentiable in time, and  differentiable twice in space.
Consider  the adjoint operator  $L^*$ of $L$ with respect to $\langle L u, v  \rangle := \int_0^T \int_M (Lu)v d\mu_t \, dt$.
 It satisfies 
$$L^*_{t,x} = \frac12 \D_t + \frac{\partial}{\partial_t} + \frac12 trace_{g(t)} (\alpha (t)) .$$
The fundamental solution $P^*(y,\tau, x, t) $ of $L^*$, satisfy :

\begin{equation}
  \left\{\begin{array}{l} 
   L^*_ {\tau,y} P^*(y,\tau,x,t) =0 \\
\lim_{\tau \nearrow t } P^*(y,\tau,x,t) = \delta_{y}
   \end{array} \right .
\end{equation}
The adjoint property yields:
 $$ P(x,t,y,\tau) =  P^*(y,\tau,x,t) . $$ 

After a time reversal, $P^*( y , t- s , x, t) $  satisfies the following heat equation :

\begin{equation}
  \left\{\begin{array}{l} 
  \partial_s P^*(y,t - s,x,t) = \frac12 \D_{g(t-s), y } P^*  + \frac12  trace_{g(t-s)} (\alpha (t-s))(y) P^* \\
\lim_{ s\searrow  0 } P^*(y,t-s ,x,t) = \delta_{y} \\
   \end{array} \right .
\end{equation}

Suppose that there exist functions  $  \tau (t)$ and  $  \underline{\tau} (t)$such that :

$$ \frac12 \sup_{y \in M }  trace_{g(t)} (\alpha (t))(y)  \le \tau (t)$$
$$ \frac12 \inf_{y \in M }  trace_{g(t)} (\alpha (t))(y)  \ge \underline{\tau}(t)$$
By Feynman Kac formula, we conclude that : 
$$   P^*(y,t - s,x,t) \le e^{\frac12 \int_0^s \tau (t-s+u) du} \overline{P} (y, s, x,t)  .$$   
We now fix $t$.
Let $\overline{P} (y, s, x,t ) $ be the fundamental solution, defined by :

\begin{equation}
  \left\{\begin{array}{l} 
  \partial_s \overline{P} (y,s,x,t) = \frac12 \D_{g(t-s), y } \overline{P} (y,s,x,t) \\
\lim_{ s\searrow  0 }  \overline{P} (y,s,x,t)  = \delta_{y} \\
   \end{array} \right .
\end{equation}
We have in particular:
$$ X^{g(\frac{t}{2}+s)}_{\frac{t}{2}}(y)      \overset{\mathcal{L}}{=}  \overline{P} (y,\frac{t}{2},x,t) \, d \mu_t(x)  $$

\begin{equation} \label{heat-back}
  \left\{\begin{array}{l} 
  \partial_s f(s,x) = \frac12 \D_{g(t-s)} f(s,x)  \\
 f(0,x) = f_0(x) \\
   \end{array} \right .
\end{equation}

\begin{thm} \label{one-diag}
Suppose that the $g(\frac{t}{2} + s)$-Brownian motion does not explode before the time $\frac{t}{2}$ and that there exists $ C \in \R$ such that
 $ \forall s \in[0, \frac{t}{2}]$:
$$ \Ric_{g( s)} - \alpha( s) \ge C g( s). $$
Assume that $g( t- s ) $-Brownian motion  does not explode before the time $\frac{t}{2}$ and that there exist $ \tilde{C} \in \R$ such that $ \forall s \in[0, \frac{t}{2}]$:
$$ \Ric_{g( t- s)} + \alpha( t - s) \ge \tilde{C} g( t-s). $$

Then the fundamental solution of \eqref{heat} that we note $P(x,t,y,0) $ satisfies :
$$ P(x,t,y,0)   \le    e \frac{e^{\frac12 \int_0^t \tau(s)\, ds}}{\Big( V_t(B_t(x,\sqrt{\frac {(\frac{t}{2})^2}{ (\frac{1- e^{-\tilde{C}\frac{t}{2}}}{\tilde{C}} )} } ) )   \Big)^{\frac{1}{2}}}           
\frac{e^{ -\frac12 \int_0^{\frac{t}{2}} \underline{\tau}(s)\, ds}}{\Big( V_{0}(B_{0}(y,\sqrt{\frac {(\frac{t}{2})^2}{ (\frac{1- e^{-C\frac{t}{2}}}{C})}} )  )    \Big)^{\frac{1}{2}}}$$
\end{thm}

\begin{proof}
By the Chapman Kolmogorov formula we have:
$$\begin{aligned}
  P(x,t,y,0) &= \int_M P(x,t,z,\frac{t}{2})     P(z,\frac{t}{2},y,0)         \, d\mu_{\frac{t}{2}}(z) \\
             &= \int_M P(x,t,z,\frac{t}{2})     P^*(y,0,z,\frac{t}{2})         \, d\mu_{\frac{t}{2}}(z) \\
             &\le \big( \int_M  (P(x,t,z,\frac{t}{2})    )^2   d\mu_{\frac{t}{2}}(z)  \big  )^{\frac12} 
 \big( \int_M  (  P^*(y,0,z,\frac{t}{2})   )^2  d\mu_{\frac{t}{2}}(z)   \big)^{\frac12} .\\
\end{aligned}$$
Recall that $ P(x,\frac{t}{2}+s ,z,\frac{t}{2})$ is the fundamental solution, that starts at $\delta_x $ at time $s=0$, of :

\begin{equation} 
  \left\{\begin{array}{l} 
  \partial_s f(s,x) = \frac12 \D_{g( \frac{t}{2}  + s)} f(s,x)  \\
 f(0,x) = f_0(x) \\
   \end{array} \right .
\end{equation}
Then we have : $$  P_{0,\frac{t}{2}}f_0(x) := f_0(\frac{t}{2},x)= \E[ f_0 (X^{t- .}_{\frac{t}{2}}(x)  )]   $$
According to the proof of  corollary \ref{corol}, we get that for $f_0 \in C_b(M) \cap L^{2}(\mu_{\frac{t}{2}})$ :
 $$  \mid P_{0,\frac{t}{2}}f_0   \mid   (x)  \le  \frac{e^{\frac{\int_0^{\frac{t}{2}}  \tau(\frac{t}{2}+s)\, ds   + 1}{2}}}{\big( V_t(B_t(x,\sqrt{\frac {(\frac{t}{2})^2}
{ (\frac{1- e^{-\tilde{C}(\frac{t}{2})}}{\tilde{C}} )} } ) )   \big)^{\frac{1}{2}}} \parallel f_0\parallel_{L^{2} (\mu_{\frac{t}{2}}) } . $$

Given $ x_0 \in M $ and $n \in \mathbb{N} $, we apply the above inequality to $ f_0(y) := P(x, t ,y,\frac{t}{2}) \wedge (n II_{B (x_0 , n)}(y) ) $ to obtain :
$$
\begin{aligned}
&  \int_{M} P(x, t ,z,\frac{t}{2}) (P(x, t ,z,\frac{t}{2})  \wedge  n II_{B (x_0 , n)} (z)) \, d\mu_{\frac{t}{2}} (z) \\
& \int_{M} \big( P(x, t ,z,\frac{t}{2}) \wedge (n II_{B (x_0 , n)} (z)) \big)^2 \, d\mu_{\frac{t}{2}} (z) \\
&\le   \frac{e^{\frac{\int_0^{\frac{t}{2}}  \tau(\frac{t}{2}+s)\, ds   + 1}{2}}}{\big( V_t(B_t(x,\sqrt{\frac {(\frac{t}{2})^2}
{ (\frac{1- e^{-\tilde{C}(\frac{t}{2})}}{\tilde{C}} )} } ) )   \big)^{\frac{1}{2}}}    \Big( 
\int_{M} \big( P(x, t ,z,\frac{t}{2}) \wedge (n II_{B (x_0 , n)} (z)) \big)^2 \, d\mu_{\frac{t}{2}} (z) \Big)^{\frac{1}{2}} . \\
\end{aligned} 
$$ 
Letting  $ n$ goes to infinity, we obtain that the heat kernel is twice integrable, and that:
$$
\begin{aligned}
\Big( 
\int_{M} \big( P(x, t ,z,\frac{t}{2}) \big)^2 \, d\mu_{\frac{t}{2}} (z) \Big)^{\frac{1}{2}} 
\le \frac{e^{\frac{\int_0^{\frac{t}{2}}  \tau(\frac{t}{2}+s)\, ds   + 1}{2}}}{\big( V_t(B_t(x,\sqrt{\frac {(\frac{t}{2})^2}
{ (\frac{1- e^{-\tilde{C}(\frac{t}{2})}}{\tilde{C}} )} } ) )   \big)^{\frac{1}{2}}} .\\
\end{aligned}
$$

Recall that:
$$   P^*(y,0,x,\frac{t}{2}) \le e^{\frac12 \int_0^{\frac{t}{2}} \tau (u) du} \overline{P} (y, \frac{t}{2}, x, \frac{t}{2})  ,$$
where $\overline{P} (y, \frac{t}{2}, x, \frac{t}{2}) $ is the heat kernel at time $\frac{t}{2} $, that start at time $ 0 $ at $ \delta_y$, of the following equation:

\begin{equation} 
  \left\{\begin{array}{l} 
  \partial_s f(s,x) = \frac12 \D_{g(\frac{t}{2}-s)} f(s,x)  \\
 f(0,x) = f_0(x) \\
   \end{array} \right .
\end{equation}

We also  have: $$ \overline{P} _{0,\frac{t}{2}}f_0(x) := f_0(\frac{t}{2},x)= \E[ f_0 (X^{ g(.)}_{\frac{t}{2}}(x)  )]   $$
Here the family of metric is $s \mapsto g(\frac{t}{2} -s) $, so  many changes of sign are involved. However, the proof of the following  is the same as the
 proof of corollary \ref{corol}. We get for $f_0 \in C_b(M) \cap L^{2}(\mu_{\frac{t}{2}})$ :
 $$  \mid \overline{P} _{0,\frac{t}{2}}f_0   \mid   (y)  \le  \frac{e^{\frac{-\int_0^{\frac{t}{2}}  \underline{\tau}(s)\, ds   + 1}{2}}}{\big( V_0(B_0(y,\sqrt{\frac {(\frac{t}{2})^2}
{ (\frac{1- e^{-C(\frac{t}{2})}}{C} )} } ) )   \big)^{\frac{1}{2}}} \parallel f_0\parallel_{L^{2} (\mu_{\frac{t}{2}}) } . $$
Similarly we can show the square integrability of the kernel and the following inequality:

$$
\begin{aligned}
\Big( 
\int_{M} \big( \overline{P}(y, \frac{t}{2}  ,z,\frac{t}{2}) \big)^2 \, d\mu_{\frac{t}{2}} (z) \Big)^{\frac{1}{2}} 
\le \frac{e^{\frac{-\int_0^{\frac{t}{2}}  \underline{\tau}(s)\, ds   + 1}{2}}}{\big( V_0(B_0(y,\sqrt{\frac {(\frac{t}{2})^2}
{ (\frac{1- e^{-C(\frac{t}{2})}}{C} )} } ) )   \big)^{\frac{1}{2}}} .\\
\end{aligned}
$$
We obtain : 
$$
\begin{aligned}
& \big( \int_M  (  P^*(y,0,z,\frac{t}{2})   )^2  d\mu_{\frac{t}{2}}(z)   \big)^{\frac12} \\
&\le  e^{\frac12 \int_0^{\frac{t}{2}} \tau (u) du}  \big( \int_M    \overline{P}^2 (y, \frac{t}{2}, z, \frac{t}{2})  \, d\mu_{\frac{t}{2}} (z) \Big)^{\frac{1}{2}} \\
&\le  e^{\frac12 \int_0^{\frac{t}{2}} \tau (u) du}   \frac{e^{\frac{-\int_0^{\frac{t}{2}}  \underline{\tau}(s)\, ds   + 1}{2}}}{\big( V_0(B_0(y,\sqrt{\frac {(\frac{t}{2})^2}
{ (\frac{1- e^{-C(\frac{t}{2})}}{C} )} } ) )   \big)^{\frac{1}{2}}}  .\\
\end{aligned}
$$

\end{proof}
\begin{remark}
 Having a heat kernel estimate for the heat equation we have simultaneously a kernel estimate of conjugate equation.

\end{remark}

\begin{remark}
If $g(t)=g(0)$ is constant, and $Ric_{g(0)} \ge 0$ we have $\tau(t) = \underline{\tau}(t) =0$, $ C =\tilde{C}=0$ and we deduce Li Yau one diagonal estimate of the usual heat equation
on complete manifolds:

$$ P_t(x,y)   \le    e \frac{1}{\big( V_0(B_0(x,\sqrt{ \frac{t}{2} } ) )   \big)^{\frac{1}{2}}}           
\frac{1}{\big( V_{0}(B_{0}(y, \sqrt{ \frac{t}{2} })  )    \big)^{\frac{1}{2}}}$$

\end{remark}

\begin{remark} \label{rmq-Ric}
 
 For the Ricci flow, in dimension $3$,  there is a result of Hamilton that  says : if the Ricci curvature is non negative at time $0$,
 it is still non negative for all time before the critical time. 
In arbitrary dimension $n$,  if we  suppose that the Ricci curvature is positive at all times, we have the following one diagonal estimate (use the above theorem \ref {one-diag} with $ C = 0 $ and $ \tilde {C} =0 $)  :
$$ P(x,t,y,0)   \le    e \frac{e^{\frac12 \int_0^t \tau(s)\, ds}}{\big( V_t(B_t(x,\sqrt{ \frac{t}{2}} ) )   \big)^{\frac{1}{2}}}           
\frac{e^{ -\frac12 \int_0^{\frac{t}{2}} \underline{\tau}(s)\, ds}}{\big( V_{0}(B_{0}(y,\sqrt{\frac{t}{2}} ))      \big)^{\frac{1}{2}}}.$$

 Recall that in the case of Ricci flow : $ \tau(s) = - \frac12 \inf_{M} R(s,.) $ and $ \underline{\tau}(s)= - \frac12 \sup_{M} R(s,.)$ so we have  :
$$ P(x,t,y,0)   \le    e \frac{e^{-\frac14 \int_{\frac{t}{2}}^{t}  \inf_{M} R(s,.) \, ds}}{\big( V_t(B_t(x,\sqrt{ \frac{t}{2}} ) )   \big)^{\frac{1}{2}}}           
\frac{e^{ \frac14 \int_0^{\frac{t}{2}} \big( \sup_{M} R(s,.) - \inf_{M} R(s,.)  \big) \, ds}}{\big( V_{0}(B_{0}(y,\sqrt{\frac{t}{2}} ))      \big)^{\frac{1}{2}}}.$$
\end{remark}

\section{Grigor'yan tricks, one diagonal estimate to Gaussian estimate, the Ricci flow case}

 In this section we  use the one diagonal estimate of  the previous section to derive a Gaussian type estimate of the heat kernel coupled with Ricci flow 
(for complete manifold with non negative Ricci curvature)
 . The proof involves in several steps. In particular, we  use a modification of  Grigor'yan tricks to control exponential integrability of the square of the heat kernel,
combined to an  adapted version of
Hamilton entropy estimate to control the difference of the heat kernel at two points.
This type of strategy, is a modification of different  arguments which  appears in the literature on the Ricci flow ( Hamilton, Lei Ni , Cao-Zhang).
 Unfortunately, we have not been able to cover  the case of general manifold (without assumption of non negativity of the Ricci curvature). 

We start this section by the following entropy estimate.
\begin{lemma} \label{Hamilton}
 Let $f$ a positive solution of \eqref{heat}, where $ \alpha_{i,j}(t)= - (\Ric_{g(t)})_{i,j}$ ,   $t \in [0, ..T_c[$ and $M_{\frac{t}{2}} :=  \sup_{x \in M} f(\frac{t}{2},x)$  then for all $x,y \in M$ 
$$ f(t,x) \le \sqrt{f(t,y)} \sqrt{M_{\frac{t}{2}}} e^{\frac{d^2_t(x,y)} {t}} . $$
\end{lemma}

\begin{proof}

By the homogeneity of the desired inequality under multiplication by a constant, and the homogeneity  of the heat equation under the same operation, we can suppose that $ f > 1 $, in the proof.

 Using an orthonormal frame and Weitzenbock formula, we have the following  identity:
 $$ ( -  \partial_t + \frac12 \D_{g(t)})   \big(  \frac{ \parallel \triangledown  f  \parallel ^{2} }{ f } \big) (x,t) =  \frac{1}{ f} \Big(
\parallel Hess f - \frac{\triangledown^{t} f \otimes \triangledown f}{f} \parallel^{2}_{HS}     +  (\Ric_t + \dot{g}) ( \triangledown {f};\triangledown {f} )       \Big)(x,t).$$
Thus,  in the case of Ricci flow we get :
$$ ( -  \partial_t + \frac12 \D_{g(t)})   \Big(  \frac{ \parallel \triangledown  f  \parallel ^{2} }{ f } \Big) \ge 0 . $$
By a direct computation we have :
$$ ( -  \partial_t + \frac12 \D_{g(t)})  ( f \log f ) (x,t) = \frac12 \frac{ \parallel \triangledown  f  \parallel ^{2} }{ f }(x,t).$$
Let $$ N_s :=  h(s)\big(  \frac{ \parallel \triangledown  f  \parallel ^{2} }{ f } \big) (t-s, X_s^t(x))\big) +  (f log f) (t-s, X_s^t(x)),$$ where $ X_s^t(x)$ is a $g(t-s)$-Brownian motion started at $x$.
If  $h(s) := \frac{t/2 -s}{2} $ then by It\^o formula, it is easy to see that $N_s$ is a super-martingale. So we have :
$$ \E[ N_0] \le \E[N_ {\frac{t}{2}}],$$ that is :
$$ \begin{aligned}
\frac{t}{4}  \frac{ \parallel \triangledown  f  \parallel ^{2} }{ f } (t,x) + (f \log f) (t,x) &\le  \E[ (f \log f) (\frac{t}{2} ,X_{\frac{t}{2}}^t(x) )   ] \\
                        &\le   \E[ f (\frac{t}{2} ,X_{\frac{t}{2}}^t(x) ) ] \log ( M_{\frac{t}{2}} ) \\
  &= f(t,x)  \log ( M_{\frac{t}{2}} ). \\
\end{aligned}$$ 
Where we have used that $ f >1 $ and that $ f(t-s,X_s^t(x) ) $ is a  martingale. The above computation yields:
$$ \frac{ \parallel \triangledown  f  \parallel  }{ f} (t,x) \le \frac{2}{\sqrt{t}} \log( \frac{ M_{\frac{t}{2}}}{ f(t,x)} ),$$  
and consequently:
$$  \parallel \triangledown  \sqrt{log (\frac{M_{\frac{t}{2}}}{f(x,t)} )  }\parallel_{t}   \le \frac{1}{\sqrt{t}}.$$ 
After integrating this inequality along a $g(t)$ geodesic  between $x$ and $y$, we get :
$$ 
\begin{aligned}
\sqrt{\log (\frac{M_{\frac{t}{2}}}{f(y,t)} ) } &\le \sqrt{\log (\frac{M_{\frac{t}{2}}}{f(x,t)} ) } + \frac{d_t(x,y)}{\sqrt{t}}, \\
\end{aligned}$$ 
that is 
$$ 
\begin{aligned}
f(t,x) \le \sqrt{f(t,y)} \sqrt{M_{\frac{t}{2}}} e^{\frac{d^2_t(x,y)} {t}} . 
\end{aligned}$$

\end{proof}

 Now, we adapt   the argument of Grigor\'yan 
 to the situation of Ricci flow (with non negative Ricci curvature).
 We begin  by recalling  Remark \ref{rmq-Ric}. 
 The assumption of the positivity of the Ricci curvature  gives $R(x,s) \ge 0$. Let $\Psi(t) := e^{\frac14 \int_0^{t}  \sup_{M} R(s,.) \, ds} $
 so the estimate in \ref{rmq-Ric} becomes :
 $$P(x,t,y,0)   \le    e \frac{\Psi(t)}{\big( V_t(B_t(x,\sqrt{ \frac{t}{2}} ) )   \big)^{\frac{1}{2}} \big( V_{0}(B_{0}(y,\sqrt{\frac{t}{2}} ))      \big)^{\frac{1}{2}}   }           
. $$ 

 \begin{lemma}
  Suppose that the family of  metrics $ g(t)$ comes from the Ricci flow, and let $B$ be a Borelian in $ M$.  Then:
    $$ \frac{1}{ V_t(B)^{\frac12}} \le   \frac{ \Psi(t) }{ V_0(B)^{\frac12}}                .   $$
Moreover if  $\Ric_{g( . )}  \ge 0 $ then for all $ x,y \in M$ and $r >0 $ we have :
 $$ \frac{1}{V_0(B_t(x,r) )^{\frac12}} \le \frac{1}{V_0(B_0(x,r) )^{\frac12}} .$$
 \end{lemma}

\begin{proof}
 A simple computation  computation shows that:
$$ \frac{d}{dt} \mu_t = \frac12 \trace_{g(t)} (\dot{g}(t)) \mu_t.$$
In the case of a Ricci flow this becomes $ \frac{d}{dt} \mu_t(dx) = - \frac12 R(x,t) \mu_t(dx).$ Thus, the first inequality of the lemma follows from a direct integration.
For the second point, it's clear that $\Ric \ge 0 $ yields that $d_t (x,y)$ is non increasing in time.
So $ B_0(x,r) \subset B_t(x,r) $, which clearly gives $ \frac{1}{V_0(B_t(x,r) )^{\frac12}} \le \frac{1}{V_0(B_0(x,r) )^{\frac12}} $.
\end{proof}

The above lemma immediately yields the following remark.

\begin{remark} \label{rem_3}
 If $\Ric_{g(.)} \ge 0$ , and $ \dot(g)(t) = - \Ric_{g(t)} $ then  we have:
 $$P(x,t,y,0)   \le    e \frac{\Psi(t)^2}{\big( V_0(B_0(x,\sqrt{ \frac{t}{2}} ) )   \big)^{\frac{1}{2}} \big( V_{0}(B_{0}(y,\sqrt{\frac{t}{2}} ))      \big)^{\frac{1}{2}}   }           
. $$
\end{remark}

\begin{proposition} \label{prop_3_1}
Let $g(t)$ be a solution of Ricci flow such that $\Ric_g(t) \ge 0 $, and  $r >0$ , $t_0>t \ge 0$, $A \ge 1 $.
Let also :
$$
  \xi(y,t)  =  \begin{cases}                 
 \frac{-(r - d_t(x,y))^2}{A (t_0-t)}      \quad  & if \quad d_t(x,y)\le r \\
  0  \quad  & if \quad  d_t(x,y)\ge r
               \end{cases}
$$
and $\Lambda (t) = \int_0^t \inf_{x \in M} (R(s,x)) ds $. 
Then for $f(t,x)$ a solution of \eqref{heat} we have for  $ t_2 < t_1 < t_0$:

$$\int_{M} f^2(t_1,y) e^{\xi(y,t_1) } \mu_{t_1} (dy )  \le e^{-(\Lambda (t_1)-\Lambda (t_2))}  \int_{M} f^2(t_2,y) e^{\xi(y,t_2) } \mu_{t_2} (dy )  . $$

\end{proposition}

\begin{proof}
 By direct computation and using intensively that $ \Ric_{g(t)} \ge 0 $.
\end{proof}

Let $$
\begin{aligned}
  I_r(t) & :=  \int_{M \setminus B_t(x,r)} f^2(t,y)\mu_{t} (dy ) \\  
\end{aligned}.
 $$

\begin{proposition}\label{prop_3_2}
 Under the same assumptions as in the above proposition, and  if, $\rho < r $  we have:
$$
\begin{aligned}
  I_r(t_1) & \le e^{-(\Lambda (t_1)-\Lambda (t_2))} \big( I_{\rho}(t_2) + e^{\frac{-(r-\rho)^2}{ A(t_1 - t_2)}   }\int_{M} f^2(t_2,y) \mu_{t_2} (dy )   \big) \\  
\end{aligned}.
 $$
\end{proposition}

\begin{proof}
  $$
\begin{aligned}
  I_r(t_1) & :=  \int_{M \setminus B_{t_1}(x,r)} f^2(t_1,y)\mu_{t_1} (dy ) \\
         & \le   \int_{M \setminus B_{t_1}(x,r)} f^2(t_1,y)  e^{\xi(y,t_1) }   \mu_{t_1} (dy ) \\
         & \le   \int_{M} f^2(t_1,y)  e^{\xi(y,t_1) }   \mu_{t_1} (dy ) \\
         & \le  e^{-(\Lambda (t_1)-\Lambda (t_2))}  \int_{M} f^2(t_2,y)  e^{\xi(y,t_2) }   \mu_{t_2} (dy ) \\
         & \le  e^{-(\Lambda (t_1)-\Lambda (t_2))}  \Big( \int_{B_{t_2}(x,\rho)} f^2(t_2,y)  e^{\xi(y,t_2) }   \mu_{t_2} (dy ) \\
 &   \quad + \int_{M \setminus B_{t_2}(x,\rho)} f^2(t_2,y)  e^{\xi(y,t_2) }   \mu_{t_2} (dy )\Big) \\
         &\le e^{-(\Lambda (t_1)-\Lambda (t_2))}  \Big( I_{\rho}(t_2) + \int_{B_{t_2}(x,\rho)} f^2(t_2,y)  e^{\xi(y,t_2) }   \mu_{t_2} (dy ) \Big) \\
 & \le e^{-(\Lambda (t_1)-\Lambda (t_2))}  \Big( I_{\rho}(t_2) + e^{\frac{-(r-\rho)^2}{ A(t_0 - t_2)}   }\int_{M} f^2(t_1,y)  \mu_{t_2} (dy ) \\
\end{aligned}.
 $$

Then remark that the definition of $I_r(t) $ is independent of $ t_0$ and of the corresponding $ \xi$, so we can pass to the limit when $ t_0 \searrow t_1$ in the following :
$$  I_r(t_1) \le    e^{-(\Lambda (t_1)-\Lambda (t_2))}  \Big( I_{\rho}(t_2) + e^{\frac{-(r-\rho)^2}{ A(t_0 - t_2)}   }\int_{M} f^2(t_1,y)  \mu_{t_2} (dy ) $$ to obtain the desired result. 
\end{proof}

We  apply the above proposition to the heat kernel $P(x,t,y,0) $ of the equation \eqref{heat_kernel} that also satisfy \eqref{heat}.
\begin{thm}
If $ \dot{g} (t) = - Ric_{g(t)}$ and $ \Ric_{g(t)} \ge 0$, and the following technical assumption is satisfied for the initial manifold
\begin{itemize}
\item H1 : if $M$ is not compact, we suppose that there exists a uniform constant $c_n > 0$ such that $Vol(B_{g(0)}(x,r)) \ge c_n r^n $ (that is a non collapsing condition )
 \item H2 : all the curvature tensor is bounded for the metric $(M,g(0)$.
\end{itemize}

Then for all $ a>1$  there exist  two  positive explicit constants $q_a $, $m_a $ depending only on $a$ and on the dimension,
 such that we have the following heat kernel estimate :
$$P(y_0,t,x_0, 0)    \le q_a\frac{ e^{ \int_0^{t} \frac58  \sup_{M} R(u,.) - \frac14 \inf_{M} R(u,.)   du }         }     
                    {\big( V_0(B_0(x_0,\sqrt{ t} ) )   \big)^{\frac{1}{2}}    V_0( B_0(y_0, \sqrt{t}))^{\frac{1}{2}}    } e^{- \frac{ d_t(x_0,y_0)^2}{  m_{a} t  } } . $$ 
\end{thm}

\begin{proof}
Let $ f(t,x) := P(x,t,y,0) $ be the heat kernel of \eqref{heat} that is the solution  of  equation \eqref{heat_kernel}. Then we have by the proof of theorem  \ref{one-diag}:
$$ 
\begin{aligned}
 \int_{M} f^2(t,x)  \mu_{t} (dx) &= \int_{M}  P^2(x,t,y,0) \mu_{t} (dx) \\
                                  &= \int_{M}  P^{*2}(y,0,x,t) \mu_{t} (dx) \\
                                  & \le e  \frac{e^{\int_0^{t} \tau (u) - \underline{\tau}(u)\, du }} {\big( V_0(B_0(y,\sqrt{t})) \big)}\\
                                  &= e  \frac{e^{\int_0^{t}  \sup_{M} R(u,.) - \inf_{M} R(u,.)   du }} {\big( V_0(B_0(y,\sqrt{t})) \big)}\\
\end{aligned}.
$$ 
 Let $ 0< \rho < r $ , $A \ge 1 $ and $ t_2 < t_1 < t_0$ then apply proposition \ref{prop_3_2} to $ f(t,x) := P(x,t,y,0) $, to get :
$$
\begin{aligned}
  I_r(t_1) & \le e^{-(\Lambda (t_1)-\Lambda (t_2))} \big( I_{\rho}(t_2) + e^{\frac{-(r-\rho)^2}{ A(t_1 - t_2)}   }\int_{M} f^2(t_2,y) \mu_{t_2} (dy )   \big) \\  
           & \le   e^{-(\Lambda (t_1)-\Lambda (t_2))} \big( I_{\rho}(t_2) + e e^{\frac{-(r-\rho)^2}{ A(t_1 - t_2)}   }    \frac{e^{\int_0^{t_2}  \sup_{M} R(u,.) - \inf_{M} R(u,.)   du }} {\big( V_0(B_0(y,\sqrt{t_2})) \big)}     \big) \\ 
\end{aligned}.
 $$

Let $ a > 1$ be a constant. Following Gregorian   we define :
$r_k := ( \frac{1}{2} + \frac{t}{k+2})r$
and $ t_k := \frac{1}{a^k}$.

Thus proposition  \ref{prop_3_2} can be applied to $r_{k+1} < r_k $  and $t_{k+1} < t_k  $, yielding to the same estimate as before : 
$$
\begin{aligned}
  I_{r_k}(t_k) & \le e^{-(\Lambda (t_k)-\Lambda (t_{k+1}))} \big( I_{r_{k+1}}(t_{k+1}) + e  e^{\frac{-( r_k- r_{k+1} )^2}{ A(t_k- t_{k+1})}   } 
 \frac{e^{\int_0^{t_{k+1}}  \sup_{M} R(u,.) - \inf_{M} R(u,.)   du }} {\big( V_0(B_0(y,\sqrt{t_{k+1}}))  \big) } \big) \\  
   & \le e^{-(\Lambda (t_k)-\Lambda (t_{k+1}))} \big( I_{r_{k+1}}(t_{k+1}) + e  e^{\frac{-( r_k- r_{k+1} )^2}{ A(t_k- t_{k+1})}   } 
 \frac{e^{\int_0^{t_{k+1}}  \sup_{M} R(u,.) du  - \Lambda (t_{k+1}) }} {\big( V_0(B_0(y,\sqrt{t_{k+1}}))  \big) } \big) \\     
\end{aligned}.
 $$
Applying  recursively this inequality, we have for all $k$ :
\begin{equation}
\begin{aligned} \label{summ}
  I_{r_0}(t_0) & \le e^{-(\Lambda (t_0)-\Lambda (t_{k+1}))}I_{r_{k+1}}(t_{k+1}) + e   \sum_{i=0}^{k}   e^{-\Lambda (t_0)}  e^{\frac{-( r_i- r_{i+1} )^2}{ A(t_i- t_{i+1})}   } 
 \frac{e^{\int_0^{t_{i+1}}  \sup_{M} R(u,.)  du }} {\big( V_0(B_0(y,\sqrt{t_{i+1}}))  \big) }  \\     
      & \le  e^{-\Lambda (t_0)-\Lambda (t_{k+1}))}I_{r_{k+1}}(t_{k+1}) + e  e^{-\Lambda (t_0)}  e^{\int_0^{t_{0}}  \sup_{M} R(u,.)    du }    \sum_{i=0}^{k}     e^{\frac{-( r_i- r_{i+1} )^2}{ A(t_i- t_{i+1})}   } 
 \frac{1}{\big( V_0(B_0(y,\sqrt{t_{i+1}}))  \big) }  \\     
\end{aligned}, 
 \end{equation}
 
where in the last inequality we use that $Ric_{g(t)} \ge 0 $ so $R(x,t) \ge 0 $.

We have $ \lim_{k \longrightarrow \infty }I_{r_{k}}(t_{k}) = 0 $, by  Levis asymptotic of heat kernel. 
 This can be seen, using a probabilistic argument: 
 use that for small $t$, $\P_x (\tau_r < t) \le C e^{-\frac {r}{2t}} $, where $\tau_r := \inf {t >0, d_t(X_t (x),x) = r} $, and
$$
\begin{aligned}
 \int_{(M \setminus B_t(x,r))} P^2(x,t,y,0)\mu_{t} (dy ) & \le  \frac{cst}{t^{\frac{n}{2}}}  \int_{(M \setminus B_t(x,r))} P(x,t,y,0)\mu_{t} (dy )\\ 
         &  \le cst  \frac{\P_x (\tau_r < t)}{ t^{\frac{n}{2}}}  \\
         &  \le cst  \frac{ e^{-\frac {r}{2t}}}{ t^{\frac{n}{2}}}  \\
\end{aligned}
$$
and the right hand side goes to $0$ when $t$ goes to $0$. (we used  H1 to get a  global bound of the heat kernel i.e. remark \ref{rem_3}). 

So we can pass to the limit when $k$ goes to infinity in equation \eqref{summ} to get :
$$  I_{r_0}(t_0)  \le e  e^{-\Lambda (t_0)}  e^{\int_0^{t_{0}}  \sup_{M} R(u,.)    du }    \sum_{i=0}^{\infty}     e^{\frac{-( r_i- r_{i+1} )^2}{ A(t_i- t_{i+1})}   } 
 \frac{1}{\big( V_0(B_0(y,\sqrt{t_{i+1}}))  \big) }  .$$

Recall that $r_i - r_{i+1} = \frac{r}{(i+3)(i+2)} $ and $t_i - t_{i+1} = \frac{t}{a^i} (1 - \frac1a )$. 
Also By Bishop-Gromov theorem in the case $\Ric \ge 0$ we have  
$$\frac{V_0(B_0(y,\sqrt{t_{i}})) }{V_0(B_0(y,\sqrt{t_{i+1}})) } \le a^{\frac{n}{2}} := c_a .$$
Iterating  the above inequality we get : 
$$\frac{V_0(B_0(y,\sqrt{t_{0}})) }{V_0(B_0(y,\sqrt{t_{i+1}})) } \le (c_a)^{i+1} . $$ 
So we have :
$$
\begin{aligned}
I_{r_0}(t_0)  &\le e \frac{ e^{-\Lambda (t_0)}  e^{\int_0^{t_{0}}\sup_{M} R(u,.)    du  }}{V_0(B_0(y,\sqrt{t_{0}}))  }     \sum_{i=0}^{\infty}     e^{\frac{-( r_i- r_{i+1} )^2}{ A(t_i- t_{i+1})}   }  (c_a)^{i+1} \\
&\le e \frac{ e^{-\Lambda (t_0)}  e^{\int_0^{t_{0}}\sup_{M} R(u,.)    du  }}{V_0(B_0(y,\sqrt{t_{0}}))  }     \sum_{i=0}^{\infty}  
 e^{\frac{-(\frac{r}{(i+3)(i+2)})^2}{ A(\frac{t}{a^i} (1 - \frac1a ))} + (i+1) log(c_a)  }   \\
&\le  e \frac{ e^{-\Lambda (t_0)}  e^{\int_0^{t_{0}}\sup_{M} R(u,.)    du  }}{V_0(B_0(y,\sqrt{t_{0}}))  }     \sum_{i=0}^{\infty}  
 e^{\frac{- a^{i+1} r^2}{ A t_0  (a-1) (i+3)^4 }           + (i+1) log(c_a)  .}   \\
\end{aligned}
$$
There exists a constant $m_{(a,A)}$ such that $ \frac{a^{i+1} }{  A (a-1) (i+3)^4 } \ge m_{(a,A)} (i+2)$, and thus we get :
$$
\begin{aligned}
I_{r_0}(t_0)  &\le e \frac{ e^{-\Lambda (t_0)}  e^{\int_0^{t_{0}}\sup_{M} R(u,.)    du  }}{V_0(B_0(y,\sqrt{t_{0}}))  } \sum_{i=0}^{\infty}  
 e^{\frac{- m_{(a,A)} r^2}{  t_0  }  (i+2)         + (i+1) log(c_a)  }   \\
&\le e \frac{ e^{-\Lambda (t_0)}  e^{\int_0^{t_{0}}\sup_{M} R(u,.)    du  }}{V_0(B_0(y,\sqrt{t_{0}}))  }  e^{\frac{- m_{(a,A)} r^2}{  t_0  }} \sum_{i=0}^{\infty}  
 e^{ -(i+1)            (   \frac{ m_{(a,A)} r^2}{  t_0  }-  log(c_a) ) } .\\
\end{aligned}
$$
If  $\frac{ m_{(a,A)} r^2}{  t_0  }-  log(c_a) )  \ge log(2)$ then  
$$ I_{r_0}(t_0)  \le e \frac{ e^{-\Lambda (t_0)}  e^{\int_0^{t_{0}}\sup_{M} R(u,.)    du  }}{V_0(B_0(y,\sqrt{t_{0}}))  }  e^{\frac{- m_{(a,A)} r^2}{  t_0  }}, $$
If $\frac{ m_{(a,A)} r^2}{  t_0  }-  log(c_a) )  < log(2)$ then
$$\begin{aligned}
 I_{r_0}(t_0)  &\le \int_{M}  P^2(x,t_0,y,0) \mu_{t_0} (dx) \\
                &= \int_{M}  P^{*2}(y,0,x,t_0) \mu_{t_0} (dx) \\
                & \le e  \frac{e^{\int_0^{t_0}  \sup_{M} R(u,.) - \inf_{M} R(u,.)   du }} {\big( V_0(B_0(y,\sqrt{t_0})) \big)}\\
                 & \le e  \frac{e^{\int_0^{t_0}  \sup_{M} R(u,.) - \inf_{M} R(u,.)   du }} {\big( V_0(B_0(y,\sqrt{t_0})) \big)}  e^{log(2) +log(c_a) - \frac{ m_{(a,A)} r^2}{  t_0  } }  . \\
\end{aligned}$$

Take $A=1$,  we have that for all $a>1$ there exists a constant $q_a := 2ea^{\frac{n}{2}}$ and $m_a:= m_{(a,1)} $ such that :
\begin{equation} \label{eqI}
 I_{r}(t) \le q_a \frac{e^{\int_0^{t}  \sup_{M} R(u,.) - \inf_{M} R(u,.)   du }}{\big( V_0(B_0(y,\sqrt{t_0})) \big)} e^{- \frac{ m_{a} r^2}{  t  } } 
\end{equation}

Let $ x_0, y_0 \in M$ such that $d_t(x_0,y_0) \ge \sqrt{t}$, let $ r := \frac{d_t(x_0,y_0)}{2}$, then by \eqref{eqI} 
(with $I_r(t)$ defined with $ f(t,x) =  P^2( x ,t,x_0, 0) $ ),  there exists
$z_0 \in B_t(y_0, \sqrt{\frac{t}{4}}) \subset M \backslash B_t(x_0, r)$ such that :
$$\begin{aligned}
 V_t( B_t(y_0, \sqrt{\frac{t}{4}}) P^2(z_0,t,x_0, 0) &\le I_r(t)\\
                 &\le q_a \frac{e^{\int_0^{t}  \sup_{M} R(u,.) - \inf_{M} R(u,.)   du }}{\big( V_0(B_0(x_0,\sqrt{t})) \big)} e^{- \frac{ m_{a} d_t(x_0,y_0)^2}{  4t  } } \\
\end{aligned}.$$

So there exists $z_0 \in B_t(y_0, \sqrt{\frac{t}{4}})$ such that:
$$ P^2(z_0,t,x_0, 0) \le q_a \frac{e^{\int_0^{t}  \sup_{M} R(u,.) - \inf_{M} R(u,.)   du }}{\big( V_0(B_0(x_0,\sqrt{t})) V_t( B_t(y_0, \sqrt{\frac{t}{4}})\big)} e^{- \frac{ m_{a} d_t(x_0,y_0)^2}{  4t  } }  $$

We denote $q_a$ a constant that depends only of the parameter $a$ and the dimensions, that possibly changes from  line by line. 
By the above lemma (comparison of volume) we have :

$$P(z_0,t,x_0, 0) \le q_a \psi(t) \frac{e^{\frac12 \int_0^{t}  \sup_{M} R(u,.) - \inf_{M} R(u,.)   du }}{\sqrt{\big( V_0(B_0(x_0,\sqrt{t})) V_0( B_0(y_0, \sqrt{\frac{t}{4}}))\big)}} e^{- \frac{ m_{a} d_t(x_0,y_0)^2}{  8t  } }  $$

We conclude the proof by using lemma \ref{Hamilton} (for $ f(t,x) := P(x,t,x_0, 0)$)  to compare the solution of the heat equation at different point. We have :

$$
\begin{aligned}
P(y_0,t,x_0, 0)  & \le \sqrt{P(z_0,t,x_0, 0)} \sqrt{ \sup_{M}  P(. ,\frac{t}{2},x_0, 0)} e^{\frac{d_t(z_0,y_0)^2}{t} }\\ 
                 & \le \sqrt{P(z_0,t,x_0, 0)} \sqrt{ \sup_{M}  P(. ,\frac{t}{2},x_0, 0)} e^{\frac14}\\ 
                 & \le q_a \sqrt{P(z_0,t,x_0, 0)} \sqrt{\frac{\Psi(\frac{t}{2})^2}{\big( V_0(B_0(x_0,\sqrt{ \frac{t}{2}} ) )   \big)^{\frac{1}{2}} \big( t     \big)^{\frac{n}{4}}   }  }  \quad \text{use rmq} \ref{rem_3} \text{and H1}  \\ 
                 & \le q_a\frac{\psi(t)^{\frac32} e^{\frac14 \int_0^{t}  \sup_{M} R(u,.) - \inf_{M} R(u,.)   du }         }     
                    {\big( V_0(B_0(x_0,\sqrt{ \frac{t}{2}} ) )   \big)^{\frac{1}{2}}    V_0( B_0(y_0, \sqrt{\frac{t}{4}}))^{\frac{1}{4}} t^{\frac{n}{8}}   } e^{- \frac{ m_{a} d_t(x_0,y_0)^2}{  16t  } } \\
                 &\le q_a\frac{\psi(t)^{\frac32} e^{\frac14 \int_0^{t}  \sup_{M} R(u,.) - \inf_{M} R(u,.)   du }         }     
                    {\big( V_0(B_0(x_0,\sqrt{ \frac{t}{2}} ) )   \big)^{\frac{1}{2}}    V_0( B_0(y_0, \sqrt{\frac{t}{2}}))^{\frac{1}{2}}    } e^{- \frac{ m_{a} d_t(x_0,y_0)^2}{  16t  } } \\
\end{aligned}
$$ 
where in the two last inequalities we use Bishop-Gromov theorem volume comparison theorem  to compare volume of ball in positive Ricci curvature case to the corresponding Euclidean volume :

$ \frac{1}{r^n} \le \frac{cst_n}{v_0(B_0(x,r))}  $ and $ \frac{1}{v_0(B_0(x, r))}  \le  \frac{cst_n(\lambda) }{v_0(B_0(x, \lambda r))}  $ for $\lambda \ge 1 $.

By the same argument we have in a more natural way :

$$P(y_0,t,x_0, 0)    \le q_a\frac{ e^{ \int_0^{t} \frac58  \sup_{M} R(u,.) - \frac14 \inf_{M} R(u,.)   du }         }     
                    {\big( V_0(B_0(x_0,\sqrt{ t} ) )   \big)^{\frac{1}{2}}    V_0( B_0(y_0, \sqrt{t}))^{\frac{1}{2}}    } e^{- \frac{ m_{a} d_t(x_0,y_0)^2}{  16t  } }  $$  

\end{proof}
 \begin{remark}
  The constant $ \frac58$ and $\frac14 $ are far from being optimal. Moreover,assumption H1 does not seem to be be necessary. Hypothesis H2 is a sufficient condition 
for the existence in short time of the Ricci flow \cite{Shi}.

 \end{remark}
\begin{remark}
 The above estimate also produce  a control for the heat kernel of the conjugate heat equation.
\end{remark}

\begin{remark}
 We could apply the same strategies for a general family of metric.
\end{remark}


\end{document}